\documentclass[11pt]{article}

\usepackage[T1]{fontenc}
\usepackage[utf8]{inputenc}
\usepackage[margin=1in]{geometry}
\usepackage{lmodern}
\usepackage{amsmath,amssymb,amsthm,mathtools}
\usepackage{xurl}
\usepackage[colorlinks=true,linkcolor=blue,citecolor=blue,urlcolor=blue]{hyperref}

\newcommand{\CC}{\mathbb{C}}
\newcommand{\OO}{\mathcal{O}}
\newcommand{\HC}{\operatorname{HC}}
\newcommand{\HD}{\operatorname{HD}}

\newcommand{\soseven}{\mathfrak{so}_7}
\newcommand{\spsix}{\mathfrak{sp}_6}

\newtheorem{proposition}{Proposition}
\newtheorem{theorem}[proposition]{Theorem}
\newtheorem{conjecture}{Conjecture}
\theoremstyle{definition}
\newtheorem{definition}{Definition}
\theoremstyle{remark}
\newtheorem{remark}{Remark}
\newenvironment{sourceproof}
  {\begingroup\begin{proof}[Source]}
  {\end{proof}\endgroup}

\title{Super-Chevalley Restriction and Relative Lie Algebra Cohomology\\
over the \texorpdfstring{$2|3$}{2|3} Algebra}
\author{%
Chi-Ming Chang$^{a,b}$\\[0.8em]
{\small $^a$ \textit{Yau Mathematical Sciences Center (YMSC), Tsinghua University, Beijing, China}}\\[0.25em]
{\small $^b$ \textit{Beijing Institute of Mathematical Sciences and Applications (BIMSA), Beijing, China}}\\[0.45em]
{\small \texttt{cmchang@tsinghua.edu.cn}}}
\date{}

\begin{document}
\maketitle

\begin{abstract}
Let
\[
A:=\CC[z_+,z_-]\otimes \Lambda(\theta_1,\theta_2,\theta_3).
\]
This algebra is supercommutative, with \(z_\pm\) even and
\(\theta_1,\theta_2,\theta_3\) odd. For a reductive Lie algebra
\(\mathfrak g\), let \(\mathfrak g[A]:=\mathfrak g\otimes A\) denote the
corresponding current Lie superalgebra. Motivated by the Chang--Yin
description of weak-coupling \(1/16\)-BPS cohomology in
\(\mathcal N=4\) super-Yang--Mills, we study the relative Lie algebra
cohomology
\[
H^\bullet(\mathfrak g[A],\mathfrak g;\CC).
\]
For this algebra, we isolate three finite-rank phenomena. First, the natural
super-commuting restriction map \(\operatorname{res}_{\mathfrak g,3|2}\),
viewed as a super analogue of Chevalley restriction and its commuting-scheme
variants, already fails to be an isomorphism for
\(\mathfrak g=\mathfrak{so}_7\); the obstruction is a non-Cartan class.
Second, the same algebra \(A\) produces explicit fortuitous classes for
\(\mathfrak{sl}_2\) and \(\mathfrak{so}_7\), giving concrete counterexamples to
naive stable-image expectations suggested by the type-\(A\)
Loday--Quillen--Tsygan theorem and its current-algebra refinements. Third, the
classical relative cohomologies for the Langlands-dual pair
\((\mathfrak{so}_7,\mathfrak{sp}_6)\) are not isomorphic. We then record the
conjectural quantum deformation of the differential expected to restore
duality, together with first-order evidence pairing the fortuitous and
non-Cartan \(\mathfrak{so}_7\) classes.
\end{abstract}
\vfill
\noindent{\footnotesize\emph{Disclosure.} Portions of the exposition and
\LaTeX\ editing were assisted by OpenAI Codex. All mathematical claims,
proofs, references, and editorial decisions were verified by the author.\par}
\newpage

\tableofcontents

\paragraph{Standing conventions.}
All algebras are over \(\CC\). Throughout, \(\mathfrak g\) denotes a
reductive Lie algebra. When \(G\) appears, it denotes a connected reductive
group with Lie algebra \(\mathfrak g\); \(\mathfrak t\subset \mathfrak g\)
denotes a Cartan subalgebra, \(W\) the Weyl group, and
\({}^L\!\mathfrak g\) the Langlands-dual Lie algebra.

If \(B\) is a commutative or supercommutative algebra, then
\(\mathfrak g[B]:=\mathfrak g\otimes B\) denotes the corresponding current Lie
algebra or Lie superalgebra. If \(B=\CC\oplus B_+\), then \(B_+\) denotes the
augmentation ideal. The symbol \(\Pi\) denotes parity reversal. If an affine
(super)scheme \(X\) carries an action of a group \(\Gamma\), then
\(\OO(X)^\Gamma\) denotes the ring of \(\Gamma\)-invariant regular functions
on \(X\). Relative cohomology \(H^\bullet(\mathfrak g[B],\mathfrak g;\CC)\)
always means relative Chevalley--Eilenberg cohomology with trivial
coefficients \(\CC\). When a graded Hopf algebra structure is present,
\(\operatorname{Prim}\) denotes its primitive part. When \(B\) is
\(\mathbb Z/2\)-graded, all symmetric and exterior constructions are
understood in the graded-commutative sense.

\section{Introduction}\label{sec:introduction}

This paper studies the supercommutative algebra
\begin{equation}\label{eq:A23}
A:=\CC[z_+,z_-]\otimes \Lambda(\theta_1,\theta_2,\theta_3)
\end{equation}
and the relative Lie algebra cohomology
\begin{equation}\label{eq:relative-main}
H^\bullet(\mathfrak g[A],\mathfrak g;\CC).
\end{equation}
In the gauge-theoretic setting of Chang--Yin~\cite{CY13}, this relative
cohomology computes the weak-coupling \(Q\)-cohomology of \(1/16\)-BPS
operators in \(\mathcal N=4\) super-Yang--Mills. The aim is to isolate
several finite-rank mathematical consequences of recent high-energy theory
results using the language of Lie algebra cohomology and invariant theory.
We begin by recalling two standard inputs that motivate the later questions.

The first input is the classical Chevalley restriction theorem, which
identifies invariant polynomial functions on \(\mathfrak g\) with
\(W\)-invariant polynomial functions on a Cartan subalgebra \(\mathfrak t\):
\[
\OO(\mathfrak g)^G\xrightarrow{\sim}\OO(\mathfrak t)^W.
\]
For \(d\ge 1\), one may replace a single element of \(\mathfrak g\) by a
commuting \(d\)-tuple and ask for the analogous restriction map
\[
\OO(\mathcal C_{\mathfrak g}^d)^G\longrightarrow \OO(\mathfrak t^{\oplus d})^W,
\qquad
\mathcal C_{\mathfrak g}^d:=\{(x_1,\dots,x_d)\in \mathfrak g^{\oplus d}\mid [x_i,x_j]=0\}.
\]
Here the set-theoretic notation is shorthand for the commuting scheme cut out
by the equations \([x_i,x_j]=0\). For the classical groups, this
commuting-scheme restriction map is known to be an isomorphism in broad
generality: Vaccarino~\cite{Vac07} proves the type \(A\) case,
Chen--Ngo~\cite{CN21} the symplectic case, and
Song--Xia--Xu~\cite{SXX23} the orthogonal case. The basic geometry of
commuting varieties goes back to Richardson~\cite{Ric79}. For \(d=2\), earlier
results of Joseph~\cite{Jos97} and Etingof--Ginzburg~\cite{EG02} identify the
reduced invariant-function ring for commuting pairs, and Gan--Ginzburg~\cite{GG06}
establish the type-\(A\) case by almost-commuting methods.

From a derived point of view,
Berest--Felder--Patotski--Ramadoss--Willwacher~\cite[Theorem~6.7 and Proposition~8.2]{BFPRW17}
identify the DG algebra of
\(G\)-invariant functions on the derived commuting scheme for a two-dimensional
abelian Lie algebra with a relative Chevalley--Eilenberg DG algebra, and they
formulate the corresponding derived Harish--Chandra map. More recently,
Li--Nadler--Yun~\cite[Theorem~1.4.10, Theorem~1.4.11, and Corollary~6.3.5]{LNY24}
obtain, conditional on a spectral Ansatz, the ordinary \(d=2\)
invariant-function formulas for both group and Lie algebra commuting schemes,
together with a description of the DG algebra of derived functions on the
derived commuting stack via Langlands duality. The present \(3|2\)
super-commuting restriction problem should be compared with these ordinary
two-variable classical and derived pictures.

A second input is the stable Loday--Quillen--Tsygan theorem, which identifies
the stable homology of \(\mathfrak{gl}_\infty(B)\) with cyclic homology:
\[
H_\bullet(\mathfrak{gl}_\infty(B))
\cong
\operatorname{Sym}\!\bigl(\HC_\bullet(B)[-1]\bigr).
\]
Equivalently, the primitive part of the Hopf algebra
\(H_\bullet(\mathfrak{gl}_\infty(B))\) is identified with
\(\HC_\bullet(B)[-1]\). After dualizing, this explains why cyclic cohomology
appears as the stable source in the finite-rank picture discussed next.

Teleman's cohomological packaging of Feigin's current-algebra construction
suggests the following finite-rank picture. For a general reductive
\(\mathfrak g\) of rank \(\ell\), with exponents \(m_1,\dots,m_\ell\), one
has the map
\begin{equation}\label{eq:intro-psi}
\Psi^*:
\operatorname{Sym}\!\left(\bigoplus_{j=1}^{\ell}\HC^{\,\bullet}_{(m_j)}(B)[-1]\right)
\longrightarrow
H^\bullet(\mathfrak g[B]).
\end{equation}
In this language, the Loday--Feigin conjectural picture predicts that
\(\Psi^*\) should be surjective, and in Loday's stronger form one asks that it
be an isomorphism.\footnote{Teleman~\cite{Tel03} gives this cohomological
packaging, distinguishes Loday's stronger isomorphism statement from Feigin's
weaker smooth-case surjectivity question, and explains that the stronger
statement already fails for \(\mathfrak{sl}_2[x,y]\), while the weaker
smooth-case surjectivity question remains open there.}

For the counterexamples below, it is useful to consider the slightly broader
stable-image expectation
\begin{equation}\label{eq:intro-stable}
H^\bullet(\mathfrak g[B])
\stackrel{?}{=}
\operatorname{im}\!\left(
\operatorname{Sym}\!\bigl(\HC^{\,\bullet}(B)[-1]\bigr)
\longrightarrow
H^\bullet(\mathfrak g[B])
\right).
\end{equation}
For \(\mathfrak g=\mathfrak{gl}_n\), the source should be viewed as the
stable cohomology \(H^\bullet(\mathfrak{gl}_\infty(B))\) rewritten via
Loday--Quillen--Tsygan. Since the source of~\eqref{eq:intro-psi} is a
distinguished graded subalgebra of the larger stable source
in~\eqref{eq:intro-stable}, surjectivity of~\eqref{eq:intro-psi} would
imply~\eqref{eq:intro-stable}.

Against this background, the present paper records three mathematical
consequences of recent high-energy theory results.
\begin{enumerate}
\item The natural \(3|2\) super-commuting analogue of the commuting-scheme
restriction problem already fails for \(\mathfrak g=\mathfrak{so}_7\); the
obstruction is the non-Cartan class.
\item The same algebra \(A\) produces explicit fortuitous classes, first
for \(\mathfrak{sl}_2\) and then for \(\mathfrak{so}_7\), giving concrete
finite-rank counterexamples to naive Loday--Feigin stable-image expectations.
\item The classical relative cohomologies for the Langlands-dual pair
\((\mathfrak{so}_7,\mathfrak{sp}_6)\) are not isomorphic, but recent work
suggests a quantum deformation of the differential for which duality may be
restored; at first order, the fortuitous \(\mathfrak{so}_7\) class maps to the
non-Cartan class.
\end{enumerate}

Here and below, the stable source for a classical series means its
infinite-rank relative cohomology. It is cyclic in type \(A\) by
Loday--Quillen--Tsygan, and dihedral for the orthogonal and symplectic
families; the derivative-multigraded refinement used below is explained in
Remark~\ref{rem:osp-dihedral-caveat}. Several
numbered results are quoted consequences of the cited papers and are followed
only by a brief source note recording provenance.

Section~\ref{sec:super-commuting} records the super-commuting schemes needed
for the first theme. Section~\ref{sec:bps-superfields} introduces the \(2|3\)
\(Q\)-complex in purely algebraic terms. Section~\ref{sec:stable-classes}
treats the stable sector and the Loday--Feigin side. Finally,
Section~\ref{sec:langlands-quantum} returns to the \(3|2\) counterexample
and the Langlands-duality and quantum-deformation story.

\section{Super-commuting \texorpdfstring{$d_B|d_F$}{dB|dF}-tuples}\label{sec:super-commuting}

We record the super analogue of the commuting-scheme restriction problem
in a form independent of the \(2|3\) specialization used below. For integers
\(d_B,d_F\ge 0\), set
\[
\mathfrak g^{d_B|d_F}:=\mathfrak g^{\oplus d_B}\oplus \Pi(\mathfrak g^{\oplus d_F}),
\]
equipped with the diagonal adjoint action of \(G\), where \(\Pi\) denotes
parity reversal. Its coordinate superalgebra is
\[
\OO(\mathfrak g^{d_B|d_F})
=
\operatorname{Sym}\!\bigl((\mathfrak g^*)^{\oplus d_B}\bigr)\otimes
\Lambda\!\bigl((\mathfrak g^*)^{\oplus d_F}\bigr).
\]
A point of the ambient superspace is written as
\[
(x_1,\dots,x_{d_B};\psi_1,\dots,\psi_{d_F}),
\]
with the \(x_i\) even and the \(\psi_a\) odd. The super-commuting scheme
\[
\mathcal C_{\mathfrak g}^{d_B|d_F}\subset \mathfrak g^{d_B|d_F}
\]
is the closed subsuperscheme defined by the equations
\[
[x_i,x_j],\qquad [x_i,\psi_a],\qquad [\psi_a,\psi_b].
\]
Equivalently, its defining ideal is generated by the coordinate functions
\[
\lambda([x_i,x_j]),\qquad
\lambda([x_i,\psi_a]),\qquad
\lambda([\psi_a,\psi_b]),
\qquad \lambda\in \mathfrak g^*.
\]
Since \(\CC\)-points do not detect the odd directions, it is useful to restate
the definition functorially: for every supercommutative algebra \(R\),
\[
\mathcal C_{\mathfrak g}^{d_B|d_F}(R)
:=
\left\{(x_i;\psi_a)\in
(\mathfrak g\otimes R)_0^{d_B}\times (\mathfrak g\otimes R)_1^{d_F}
\;\middle|\;
[x_i,x_j]=[x_i,\psi_a]=[\psi_a,\psi_b]=0\right\}.
\]
Here the bracket is the superbracket on \(\mathfrak g\otimes R\) induced by
the Lie bracket on \(\mathfrak g\) and multiplication in \(R\). The two
descriptions agree: an \(R\)-point of \(\mathfrak g^{d_B|d_F}\) factors
through \(\mathcal C_{\mathfrak g}^{d_B|d_F}\) exactly when the pullbacks of
the defining equations vanish, equivalently when the three families of
brackets above vanish in \(\mathfrak g\otimes R\).

The Cartan-side superspace is
\[
\mathfrak t^{d_B|d_F}:=\mathfrak t^{\oplus d_B}\oplus \Pi(\mathfrak t^{\oplus d_F})
\subset
\mathcal C_{\mathfrak g}^{d_B|d_F}.
\]
Since \(\mathfrak t\) is abelian, restriction along the inclusion gives a
natural map
\[
\operatorname{res}_{\mathfrak g,d_B|d_F}:
\OO(\mathcal C_{\mathfrak g}^{d_B|d_F})^G
\longrightarrow
\OO(\mathfrak t^{d_B|d_F})^W.
\]
The naive super-Chevalley restriction statement is that
\(\operatorname{res}_{\mathfrak g,d_B|d_F}\) should be an isomorphism. The
only case needed later is \((d_B,d_F)=(3,2)\), for which we abbreviate
\[
\mathcal C_{\mathfrak g}^{3|2},
\qquad
\mathfrak t^{3|2},
\qquad
\operatorname{res}_{\mathfrak g,3|2}.
\]

\begin{definition}[Non-Cartan class]
For any \(d_B,d_F\), a nonzero class in the kernel of
\[
\operatorname{res}_{\mathfrak g,d_B|d_F}:
\OO(\mathcal C_{\mathfrak g}^{d_B|d_F})^G
\longrightarrow
\OO(\mathfrak t^{d_B|d_F})^W
\]
will be called a \emph{non-Cartan class}.
\end{definition}

\begin{theorem}\label{prop:super-chevalley}
For \((d_B,d_F)=(3,2)\) and \(\mathfrak g=\soseven\), there exist nontrivial
non-Cartan classes. In particular, the map
\(\operatorname{res}_{\mathfrak g,3|2}\) is not an isomorphism.
\end{theorem}

\begin{sourceproof}
This statement is proved by Chang--Lin~\cite[\S 3 and Table~1]{CL25}. In the
terminology of the present paper, their argument detects a nonzero class
killed by the restriction map \(\operatorname{res}_{\mathfrak g,3|2}\), hence a
non-Cartan class.
\end{sourceproof}

\begin{remark}
Section~\ref{subsec:classical-mismatch-level18} returns to
Theorem~\ref{prop:super-chevalley} by recording an explicit representative
\(\Xi_{nc}^{\mathfrak{so}_7}\) for the non-Cartan class detected
in~\cite{CL25}. This is the class later denoted by
\([\Xi_{nc}^{\mathfrak{so}_7}]\).
\end{remark}

\section{BPS superfields, BPS words, and the supercharge}\label{sec:bps-superfields}

Fix the \(2|3\) supercommutative algebra
\[
A=\CC[z_+,z_-]\otimes \Lambda(\theta_1,\theta_2,\theta_3),
\]
and the augmentation ideal
\[
A_+:=(z_+,z_-,\theta_1,\theta_2,\theta_3)\subset A.
\]
We rewrite the \(Q\)-complex of BPS words in purely algebraic terms.
The terminology is borrowed from the high-energy theory literature, but the
constructions themselves involve only the relative current algebra
\(\mathfrak g[A]\). By a \emph{BPS superfield} we mean an odd formal variable
valued in \(\mathfrak g\otimes A_+\), equivalently an element of the
parity-reversed superspace \(\Pi(\mathfrak g\otimes A_+)\). Thus we write
\[
\Psi(z,\theta)\in \Pi(\mathfrak g\otimes A_+) \subset \Pi(\mathfrak g[A]),
\qquad
\Psi|_{z=\theta=0}=0.
\]
The vanishing at the origin is exactly the statement that only the augmentation
ideal \(A_+\) appears.
The space of \emph{BPS words} of length \(p\) is
\[
\mathcal W^p(\mathfrak g,A)
:=
\operatorname{Hom}_{\mathfrak g}\!\bigl(\Lambda^p(\mathfrak g\otimes A_+),\CC\bigr).
\]
For \(c\in \mathcal W^p(\mathfrak g,A)\), write
\[
\mathcal O_c:=\frac{1}{p!}\,c(\Psi,\dots,\Psi).
\]
\begin{proposition}\label{prop:bps-words-relative}
For every \(p\ge 0\), there is a canonical isomorphism
\begin{equation}\label{eq:phi-relative}
\Phi:
\mathcal W^p(\mathfrak g,A)\xrightarrow{\sim} C^p(\mathfrak g[A],\mathfrak g;\CC).
\end{equation}
\end{proposition}

\begin{proof}
This is immediate from the definition
\[
C^p(\mathfrak g[A],\mathfrak g;\CC)
=
\operatorname{Hom}_{\mathfrak g}\!\bigl(\Lambda^p(\mathfrak g[A]/\mathfrak g),\CC\bigr),
\]
together with the canonical identification
\(\mathfrak g[A]/\mathfrak g\cong \mathfrak g\otimes A_+\) coming from
\(A=\CC\oplus A_+\). \qedhere
\end{proof}

Thus BPS words are relative cochains written in superfield notation. We next
identify the BPS differential \(Q\) with the relative Chevalley--Eilenberg
differential. Define the \emph{supercharge} \(Q\) by
\begin{equation}\label{eq:q-psi-main}
Q\Psi(z,\theta)=\Psi(z,\theta)^2=\frac12\{\Psi(z,\theta),\Psi(z,\theta)\},
\end{equation}
and extend it to products by the graded Leibniz rule
\[
Q(XY)=Q(X)Y+(-1)^{|X|}XQ(Y),
\]
where \(|X|\in \mathbb Z/2\) denotes the parity of the homogeneous element \(X\).
To check signs, choose a homogeneous basis \(\{\tau_\alpha\}\) of
\(\mathfrak g\otimes A_+\), and let \(\{\xi^\alpha\}\) denote the corresponding
parity-shifted Chevalley--Eilenberg generators, so that
\(|\xi^\alpha|=|\tau_\alpha|+1\). Equivalently, the \(\xi^\alpha\) are
coordinates on \(\Pi(\mathfrak g\otimes A_+)\). Write
\[
[\tau_\alpha,\tau_\beta]=f_{\alpha\beta}{}^\gamma\tau_\gamma,
\qquad
\Psi=\sum_\alpha \Psi^\alpha\tau_\alpha,
\qquad
\Psi^\alpha=\xi^\alpha(\Psi).
\]
The suspension convention fixes the Koszul signs in the quadratic coefficients:
we define \(c_{\alpha\beta}{}^\gamma\) by
\[
d\xi^\gamma
=
-\frac12 c_{\alpha\beta}{}^\gamma\xi^\alpha\xi^\beta .
\]
Equivalently, \(c_{\alpha\beta}{}^\gamma\) is obtained from
\(f_{\alpha\beta}{}^\gamma\) by the standard parity-suspension signs; in the
purely even case \(c_{\alpha\beta}{}^\gamma=f_{\alpha\beta}{}^\gamma\).
With this convention, a relative \(p\)-cochain \(c\) can be written as
\[
c=\frac{1}{p!}\,c_{\alpha_1\cdots \alpha_p}\xi^{\alpha_1}\cdots \xi^{\alpha_p},
\]
and we associate to it the BPS word
\[
\mathcal O_c=\frac{1}{p!}\,c_{\alpha_1\cdots \alpha_p}\Psi^{\alpha_1}\cdots \Psi^{\alpha_p}.
\]

\begin{proposition}\label{prop:q-vs-d}
With the Chevalley--Eilenberg sign convention used in this paper, the canonical
identification~\eqref{eq:phi-relative} intertwines the supercharge \(Q\) with
minus the relative Chevalley--Eilenberg differential. Equivalently, for every
relative cochain \(c\), one has
\begin{equation}\label{eq:q-minus-d-main}
Q\mathcal O_c=-\mathcal O_{dc}.
\end{equation}
\end{proposition}

\begin{proof}
It is enough to check the intertwining on degree-one generators. Writing
\eqref{eq:q-psi-main} in the suspended coordinates above gives
\[
Q\Psi^\gamma=\frac12 c_{\alpha\beta}{}^\gamma\Psi^\alpha\Psi^\beta.
\]
On the other hand, by the defining Chevalley--Eilenberg sign convention for
the suspended structure constants,
\[
d\xi^\gamma=-\frac12 c_{\alpha\beta}{}^\gamma\xi^\alpha\xi^\beta.
\]
Under the identification~\eqref{eq:phi-relative}, the generators
\(\xi^\alpha\) correspond to \(\Psi^\alpha\). Since both \(Q\) and \(-d\) are
graded derivations and agree on generators, they agree on all cochains, which
is exactly~\eqref{eq:q-minus-d-main}.
\end{proof}

Consequently,
\begin{equation}\label{eq:qcohomology-relative}
H^\bullet\!\bigl(\mathcal W^\bullet(\mathfrak g,A),Q\bigr)
\cong
H^\bullet(\mathfrak g[A],\mathfrak g;\CC).
\end{equation}
This is the precise algebraic meaning of the statement that \(Q\)-cohomology is
relative Lie algebra cohomology: the identification is already present at the
cochain level through \(\Phi\).

In addition to the cohomological degree, the complex
\(\mathcal W^\bullet(\mathfrak g,A)\) carries a natural
\(\mathbb Z_{\ge 0}^5\)-grading by the numbers of superspace derivatives.
Concretely, each coefficient
\[
\left.
\partial_{z_+}^{r_+}\partial_{z_-}^{r_-}
\partial_{\theta_1}^{s_1}\partial_{\theta_2}^{s_2}\partial_{\theta_3}^{s_3}
\Psi(z,\theta)\right|_{z=\theta=0}
\]
is assigned derivative multidegree \((r_+,r_-,s_1,s_2,s_3)\), and the
multidegree of a word is defined additively on products. We write a general
derivative multidegree as
\[
\mathbf{n}=(n_{z_+},n_{z_-},n_{\theta_1},n_{\theta_2},n_{\theta_3})
\in \mathbb Z_{\ge 0}^5,
\qquad
|\mathbf{n}|:=n_{z_+}+n_{z_-}+n_{\theta_1}+n_{\theta_2}+n_{\theta_3}.
\]
Thus
\[
\mathcal W^\bullet(\mathfrak g,A)
=
\bigoplus_{\mathbf{n}\in \mathbb Z_{\ge 0}^5}
\mathcal W^{\bullet;\mathbf{n}}(\mathfrak g,A).
\]
The multidegree records total numbers of derivatives: the first two entries
count \(z_\pm\)-derivatives, and the last three entries count
\(\theta_i\)-derivatives. Differentiating the identity \(Q\Psi=\Psi^2\) shows
that \(Q\) sends any derivative of \(\Psi\) to a sum of products of derivatives
with the same total derivative multidegree. Equivalently, the total numbers
of
\[
\partial_{z_+},\qquad \partial_{z_-},\qquad
\partial_{\theta_1},\qquad \partial_{\theta_2},\qquad \partial_{\theta_3}
\]
are preserved by \(Q\). Hence \(Q\) preserves this multigrading, and the
relative cohomology inherits the corresponding decomposition
\begin{equation}\label{eq:derivative-refinement}
H^\bullet(\mathfrak g[A],\mathfrak g;\CC)
=
\bigoplus_{\mathbf{n}\in \mathbb Z_{\ge 0}^5}
H^{\bullet;\mathbf{n}}
(\mathfrak g[A],\mathfrak g;\CC).
\end{equation}
We also define the corresponding weighted level by
\[
L
:=
3n_{z_+}+3n_{z_-}+2n_{\theta_1}+2n_{\theta_2}+2n_{\theta_3}.
\]
This is the grading used in the later examples: \(\mathbf{n}=(0,0,4,4,4)\)
has level \(24\), while \(\mathbf{n}=(0,0,3,3,3)\) has level \(18\).
Because \(\Psi|_{z=\theta=0}=0\), every \(\Psi\) in a BPS word contributes at
least one superspace derivative. Hence for fixed multidegree \(\mathbf{n}\),
the cohomological degree satisfies
\begin{equation}\label{eq:top-degree-bound}
p\le |\mathbf{n}|.
\end{equation}
When equality holds in~\eqref{eq:top-degree-bound}, only first derivatives of
\(\Psi\) can occur.

\begin{proposition}[Top-degree/Koszul identification]\label{prop:top-degree-koszul}
Fix a multidegree \(\mathbf{n}\), and set \(p=|\mathbf{n}|\).
Then
\begin{equation}\label{eq:top-degree-supercommuting}
H^{p;\mathbf{n}}(\mathfrak g[A],\mathfrak g;\CC)
\cong
\OO(\mathcal C_{\mathfrak g}^{3|2})^G_{\mathbf{n}}.
\end{equation}
On the right, the subscript \(\mathbf n\) refers to the multihomogeneous piece
in which the two odd coordinates corresponding to \(\partial_{z_\pm}\Psi\)
carry degrees \(n_{z_\pm}\), and the three even coordinates corresponding to
\(\partial_{\theta_i}\Psi\) carry degrees \(n_{\theta_i}\).
\end{proposition}

\begin{proof}
Via~\eqref{eq:qcohomology-relative}, it suffices to compute the
\(Q\)-cohomology in this fixed multidegree. By the definition of \(p\), each
of the \(p\) factors in a word contributes exactly one superspace derivative.
Thus only the five first derivatives
\[
\partial_{z_+}\Psi,\qquad \partial_{z_-}\Psi,\qquad
\partial_{\theta_1}\Psi,\qquad \partial_{\theta_2}\Psi,\qquad \partial_{\theta_3}\Psi
\]
can appear. Since \(\Psi\) is odd, the \(\theta_i\)-derivatives are even while
the \(z_\pm\)-derivatives are odd, so after the standard parity bookkeeping
these generators form a \(3|2\)-tuple in the sense of
Section~\ref{sec:super-commuting}.

It follows that the top cochain group in this multidegree is exactly the
\(\mathfrak g\)-invariant multihomogeneous piece of the polynomial
superalgebra on that \(3|2\)-tuple. Moreover, \eqref{eq:top-degree-bound}
gives
\[
\mathcal W^{p+1;\mathbf{n}}(\mathfrak g,A)=0,
\]
so the cohomology in degree \(p\) is the quotient of this top cochain group by
the image of the incoming differential
\[
Q:
\mathcal W^{p-1;\mathbf{n}}(\mathfrak g,A)
\longrightarrow
\mathcal W^{p;\mathbf{n}}(\mathfrak g,A).
\]
An element of
\(\mathcal W^{p-1;\mathbf{n}}(\mathfrak g,A)\)
has exactly one second-derivative factor and \(p-2\) first-derivative factors.
Applying \(Q\) to that second-derivative factor and using \(Q\Psi=\Psi^2\), one
obtains second derivatives of \(\Psi^2\). After evaluation at \(z=\theta=0\),
every term containing an undifferentiated \(\Psi\) vanishes because
\(\Psi|_{z=\theta=0}=0\), and the surviving terms are precisely the brackets
of first derivatives. Therefore the image of the incoming differential is
generated by the quadratic relations
\[
[\partial_{\theta_i}\Psi,\partial_{\theta_j}\Psi],\qquad
[\partial_{\theta_i}\Psi,\partial_{z_\pm}\Psi],\qquad
[\partial_{z_{\dot\alpha}}\Psi,\partial_{z_{\dot\beta}}\Psi],
\]
where \(\dot\alpha,\dot\beta\in\{\dot 1,\dot 2\}\) and
\(z_{\dot 1}=z_+\), \(z_{\dot 2}=z_-\). These are exactly the defining
equations of \(\mathcal C_{\mathfrak g}^{3|2}\) once one identifies the three
even coordinates with \(\partial_{\theta_i}\Psi\) and the two odd coordinates
with \(\partial_{z_\pm}\Psi\). Since each multihomogeneous piece is
finite-dimensional and \(G\) is reductive, taking \(G\)-invariants commutes
with quotienting by this \(G\)-stable image. Quotienting by the image of \(Q\)
therefore gives the stated multihomogeneous piece of
\(\OO(\mathcal C_{\mathfrak g}^{3|2})^G\).
\end{proof}

Proposition~\ref{prop:top-degree-koszul} is the point at which the
super-commuting geometry of Section~\ref{sec:super-commuting} re-enters the
relative \(Q\)-complex.

\begin{remark}[Comparison with derived commuting schemes]
For an ordinary two-dimensional abelian Lie algebra \(a\),
Berest--Felder--Patotski--Ramadoss--Willwacher~\cite[Theorem~6.7 and Proposition~8.2]{BFPRW17}
construct the derived
Harish--Chandra map and identify its source DG algebra with a relative
Chevalley--Eilenberg DG algebra, while
Li--Nadler--Yun~\cite[Corollary~6.3.5 and Theorem~6.4.1]{LNY24} compute,
conditional on a spectral Ansatz, the DG algebra of derived functions on the
derived commuting stack and recover the corresponding \(d=2\) Lie-algebra
invariant-function formula.
Proposition~\ref{prop:top-degree-koszul} is not an instance of these theorems:
the present algebra \(A\) is supercommutative, and the argument here passes to
the classical top-degree quotient rather than retaining the full derived
enhancement. It should nevertheless be viewed as a close classical super
analogue. In the top derivative degree, the incoming differential imposes
exactly the super-commuting relations and produces
\(\OO(\mathcal C_{\mathfrak g}^{3|2})^G\).
\end{remark}

\section{Stable classes, the Loday--Quillen--Tsygan theorem, and fortuitous classes}\label{sec:stable-classes}

\subsection{The stable case: Loday--Quillen--Tsygan and single-graviton classes}\label{subsec:stable-case}

We begin with the stable relative type-\(A\) theorem for the
algebra~\eqref{eq:A23}. Passing from absolute to relative cohomology via the
augmentation splitting \(A=\CC\oplus A_+\), the Loday--Quillen--Tsygan theorem
gives
\begin{equation}\label{eq:relative-lqt}
H^\bullet(\mathfrak{gl}_\infty[A],\mathfrak{gl}_\infty;\CC)
\cong
\operatorname{Sym}\!\bigl(\widetilde{\HC}^{\,\bullet}(A)[-1]\bigr).
\end{equation}
Equivalently,
\begin{equation}\label{eq:relative-prim}
\operatorname{Prim}\,H^\bullet(\mathfrak{gl}_\infty[A],\mathfrak{gl}_\infty;\CC)
\cong
\widetilde{\HC}^{\,\bullet}(A)[-1].
\end{equation}
Thus the primitive stable classes are identified with reduced cyclic cohomology.
We next write explicit superfield representatives for these primitive
generators. Teleman's map~\eqref{eq:intro-psi} is the cohomological packaging
of Feigin's construction that carries this stable source into current-algebra
cohomology, while the broader stable-image expectation~\eqref{eq:intro-stable}
asks whether fixed-rank cohomology is exhausted by the image of this stable
sector.

\begin{definition}[Single-graviton and multi-graviton classes]
Following Chang--Yin~\cite[\S 2]{CY13}, fix
\[
p_\pm,m_1,m_2,m_3\in\mathbb Z_{\ge 0},
\qquad
e_1,e_2,e_3,k_\pm\in\{0,1\},
\]
and, for notational compactness, write
\[
\Psi_+:=\partial_{z_+}\Psi,\qquad
\Psi_-:=\partial_{z_-}\Psi,\qquad
\Psi_i:=\partial_{\theta_i}\Psi.
\]
For homogeneous factors \(X_1,\dots,X_n\), each of definite parity, define
\[
\operatorname{SymTr}(X_1,\dots,X_n)
:=
\frac{1}{n!}\sum_{\sigma\in S_n}\epsilon(\sigma;X)\,
\operatorname{Tr}\!\bigl(X_{\sigma(1)}\cdots X_{\sigma(n)}\bigr),
\]
where \(\epsilon(\sigma;X)\) is the Koszul sign of the permutation. Set
\begin{equation}\label{eq:single-graviton-representatives}
\begin{aligned}
\mathcal G^{k_+,k_-;m_1,m_2,m_3}_{p_+,p_-;e_1,e_2,e_3}
:={}&\left.
\partial_{z_+}^{p_+}\partial_{z_-}^{p_-}
\partial_{\theta_1}^{e_1}\partial_{\theta_2}^{e_2}\partial_{\theta_3}^{e_3}
\operatorname{SymTr}\!\Bigl(
\Psi_+^{k_+},
\Psi_-^{k_-},
\Psi_1^{m_1},
\Psi_2^{m_2},
\Psi_3^{m_3}
\Bigr)\right|_{z=\theta=0}.
\end{aligned}
\end{equation}
Although the target is stable relative cohomology, each class represented
by~\eqref{eq:single-graviton-representatives} already appears at sufficiently
large finite rank, so \(\operatorname{Tr}\) may be read as the ordinary matrix
trace in that range. A \emph{single-graviton class} is a cohomology class
represented by an expression~\eqref{eq:single-graviton-representatives}. A
\emph{multi-graviton class} is a finite product of single-graviton classes.
\end{definition}

The representative
\(\mathcal G^{k_+,k_-;m_1,m_2,m_3}_{p_+,p_-;e_1,e_2,e_3}\)
in~\eqref{eq:single-graviton-representatives} has multidegree
\[
(p_++k_+,\,p_-+k_-,\,e_1+m_1,\,e_2+m_2,\,e_3+m_3),
\]
hence level
\[
L=
3(p_++k_+)+3(p_-+k_-)+2(e_1+m_1)+2(e_2+m_2)+2(e_3+m_3).
\]

\begin{proposition}\label{prop:primitive-single-graviton}
The single-graviton classes span
\[
\operatorname{Prim}\,H^\bullet(\mathfrak{gl}_\infty[A],\mathfrak{gl}_\infty;\CC).
\]
\end{proposition}

\begin{proof}
By~\eqref{eq:relative-prim}, it is enough to exhibit a spanning family of
\(\widetilde{\HC}^{\,\bullet}(A)\) whose images are the classes
\eqref{eq:single-graviton-representatives}. The algebra
\[
A=\CC[z_+,z_-]\otimes \Lambda(\theta_1,\theta_2,\theta_3)
\]
is a free supercommutative algebra. Teleman's description of the Hodge
decomposition~\cite[\S 1]{Tel03} therefore identifies
\(\widetilde{\HC}^{\,\bullet}(A)\) with the cohomology of the truncated super
de~Rham complex; compare also Fishel--Grojnowski--Teleman~\cite{FGT08} for
the strong Macdonald theorem behind this Hodge-theoretic package, and the
ordinary smooth-commutative case in
Loday~\cite[Theorem~3.4.12]{Lod98}. In particular, every reduced cyclic class
is represented by a linear combination of monomials
\[
z_+^{p_+}z_-^{p_-}\theta_1^{e_1}\theta_2^{e_2}\theta_3^{e_3}
dz_+^{k_+}dz_-^{k_-}d\theta_1^{m_1}d\theta_2^{m_2}d\theta_3^{m_3},
\qquad
p_\pm,m_i\in\mathbb Z_{\ge 0},
\quad
e_i,k_\pm\in\{0,1\}.
\]
Here the restrictions \(e_i,k_\pm\in\{0,1\}\) reflect that
\(\theta_i\) and \(dz_\pm\) are odd, whereas the exponents \(m_i\) are
arbitrary because the \(d\theta_i\) are even and hence generate a symmetric
algebra.

Now fix one such monomial and let \([\omega]\in \widetilde{\HC}^{\,\bullet}(A)\)
denote its class. Under Feigin's construction in Teleman's cohomological
packaging~\cite{Fei88,Tel03}, the class \([\omega]\) maps to the cohomology
class represented by
\[
\left.
\partial_{z_+}^{p_+}\partial_{z_-}^{p_-}
\partial_{\theta_1}^{e_1}\partial_{\theta_2}^{e_2}\partial_{\theta_3}^{e_3}
\operatorname{SymTr}\!\bigl(
\Psi_+^{k_+},
\Psi_-^{k_-},
\Psi_1^{m_1},
\Psi_2^{m_2},
\Psi_3^{m_3}
\bigr)\right|_{z=\theta=0}.
\]
This is precisely the single-graviton representative
\[
\mathcal G^{k_+,k_-;m_1,m_2,m_3}_{p_+,p_-;e_1,e_2,e_3}.
\]
Concretely, the differential factors \(dz_\pm,d\theta_i\) determine the entries
inserted into the symmetrized trace, while the polynomial factors
\(z_\pm,\theta_i\) are extracted by the external derivatives
\(\partial_{z_\pm},\partial_{\theta_i}\) followed by evaluation at the origin.
Since the classes of the monomials \(\omega\) span
\(\widetilde{\HC}^{\,\bullet}(A)\), the corresponding single-graviton classes
span \(\operatorname{Prim}\,H^\bullet(\mathfrak{gl}_\infty[A],\mathfrak{gl}_\infty;\CC)\).
\end{proof}

\begin{remark}[Orthogonal and symplectic stable analogues]\label{rem:osp-dihedral-caveat}
There are also stable orthogonal and symplectic analogues of the
Loday--Quillen--Tsygan theorem. For an ordinary involutive associative algebra
\(B\), the stable primitive homology of the orthogonal and symplectic Lie
algebras is governed by dihedral homology rather than ordinary cyclic homology;
see Loday--Procesi~\cite{LP88} and
Loday~\cite[Theorem~10.5.5 and Corollary~10.5.6]{Lod98}. In the stable range
one has
\[
\operatorname{Prim} H_\bullet(\mathfrak{so}_\infty(B))\cong HD_{\bullet-1}(B),
\qquad
\operatorname{Prim} H_\bullet(\mathfrak{sp}_\infty(B))\cong HD_{\bullet-1}(B).
\]
Here the shift by \(-1\) parallels the one-step suspension pattern familiar
from the type-\(A\) Loday--Quillen--Tsygan theorem.
After dualizing and passing to relative cohomology via the augmentation
splitting \(B=\CC\oplus B_+\), this suggests the stable relative source
\[
H^\bullet(\mathfrak{so}_\infty[B],\mathfrak{so}_\infty;\CC)
\cong
\operatorname{Sym}\!\bigl(\widetilde{\HD}^{\,\bullet}(B)[-1]\bigr),
\qquad
H^\bullet(\mathfrak{sp}_\infty[B],\mathfrak{sp}_\infty;\CC)
\cong
\operatorname{Sym}\!\bigl(\widetilde{\HD}^{\,\bullet}(B)[-1]\bigr),
\]
where \(\widetilde{\HD}^{\,\bullet}(B)\) denotes reduced dihedral cohomology
with respect to the chosen involution. Thus the type-\(A\) cyclic source is
replaced by a dihedral source for orthogonal and symplectic algebras. In the
present \(2|3\) algebra with trivial involution, the derivative multigrading
refinement is formal. Let \(T=(\CC^\times)^5\) act on the algebra
\(A\) of~\eqref{eq:A23} by independently scaling
\(z_+,z_-,\theta_1,\theta_2,\theta_3\). The trivial involution and the
augmentation \(A\to\CC\) are \(T\)-equivariant, and the Loday--Procesi stable
maps are natural for morphisms of involutive algebras. Hence the stable
dihedral identifications are \(T\)-equivariant. Since \(T\) is diagonalizable,
they decompose into \(T\)-weight spaces, which are precisely the derivative
multidegrees \(\mathbf{n}\).
After dualizing and passing to relative cohomology through the augmentation
splitting \(A=\CC\oplus A_+\), the primitive stable relative source therefore
refines weight by weight:
\[
\operatorname{Prim}
H^{\bullet;\mathbf{n}}(\mathfrak{so}_\infty[A],\mathfrak{so}_\infty;\CC)
\cong
\widetilde{\HD}^{\,\bullet}_{\mathbf{n}}(A)[-1],
\qquad
\operatorname{Prim}
H^{\bullet;\mathbf{n}}(\mathfrak{sp}_\infty[A],\mathfrak{sp}_\infty;\CC)
\cong
\widetilde{\HD}^{\,\bullet}_{\mathbf{n}}(A)[-1].
\]
Here the subscript denotes the \(T\)-weight piece. Equivalently, the
corresponding primitive classes come from the dihedral summand of the same
de~Rham-type generators that appear in the type-\(A\) cyclic description.
\end{remark}

\subsection{The finite-rank case: fortuitous classes}\label{subsec:finite-rank-fortuitous}

We pass from the stable classes of Section~\ref{subsec:stable-case} to
fixed-rank relative cohomology. For each classical series, the standard
upper-left-block inclusions induce restriction maps from stable rank to finite
rank. Their cokernels measure exactly the failure of the stable sector to
exhaust the finite-rank cohomology; the corresponding classes are the
fortuitous ones.

\begin{definition}[Fortuitous classes]
Let \(\mathfrak g_n\) be one of the classical Lie algebras
\[
\mathfrak{gl}_n,\qquad
\mathfrak{sl}_n,\qquad
\mathfrak{so}_n,\qquad
\mathfrak{sp}_{2n}.
\]
Write \(\mathfrak g_\infty\) for the stable classical Lie algebra in the same
series. The standard upper-left-block inclusions induce a canonical restriction
map
\[
\rho_{\mathfrak g_n}:
H^\bullet(\mathfrak g_\infty[A],\mathfrak g_\infty;\CC)
\longrightarrow
H^\bullet(\mathfrak g_n[A],\mathfrak g_n;\CC).
\]
We define the space of \emph{fortuitous classes} by
\[
\operatorname{Fort}^\bullet(\mathfrak g_n,A)
:=
\operatorname{coker}(\rho_{\mathfrak g_n}).
\]
Since \(\rho_{\mathfrak g_n}\) respects the derivative
refinement~\eqref{eq:derivative-refinement}, the cokernel inherits the
corresponding multigrading; we write
\[
\operatorname{Fort}^{\bullet;\mathbf{n}}(\mathfrak g_n,A)
\]
for its multihomogeneous pieces. Thus a \emph{fortuitous class} is precisely a
finite-rank relative cohomology class not lying in the image of
\(\rho_{\mathfrak g_n}\).
\end{definition}

Accordingly, in type \(A\), \eqref{eq:relative-lqt} identifies the source of
\(\rho_{\mathfrak g_n}\) with the reduced cyclic-cohomology sector. In the
orthogonal and symplectic families, Remark~\ref{rem:osp-dihedral-caveat}
identifies the analogous derivative-multigraded dihedral sector in the stable
range. Thus fortuitous classes measure the failure to lie in the stable source:
the cyclic source in type \(A\), and the dihedral source in the present graded
orthogonal/symplectic stable picture.\footnote{From the physics side, several large-\(N\) studies
of the superconformal index~\cite{CBCMM19,BM20,CKKN24} point in the same
direction. At weak coupling, that index is the Euler characteristic of the
\(Q\)-cohomology, so those asymptotics provide counterevidence to the naive
stable-image expectation~\eqref{eq:intro-stable}. This remains only an
index-theoretic indication, not a direct theorem about the cohomology map.}

The first explicit fortuitous class for the \(2|3\) algebra occurs for
\(\mathfrak{sl}_2\), as exhibited by Chang--Lin~\cite{CL23}.

\begin{proposition}[Chang--Lin~\cite{CL23}]\label{prop:cl23}
For \(A=\CC[z_+,z_-]\otimes \Lambda(\theta_1,\theta_2,\theta_3)\) and
\(\mathfrak g=\mathfrak{sl}_2\), the space
\[
\operatorname{Fort}^{7;(0,0,4,4,4)}(\mathfrak g,A)
\]
contains nontrivial elements.
\end{proposition}

\begin{sourceproof}
This is the first explicit fortuitous class isolated in~\cite[\S 4.3]{CL23}.
In the grading conventions of the present paper, the representative written
below is a word with seven factors of \(\Psi\), hence it lies in cohomological
degree \(7\), and its derivative multidegree is \((0,0,4,4,4)\). For the
present paper, the essential point is the existence of a nonzero class in this
multigraded piece of the cokernel.
\end{sourceproof}

\begin{remark}
By construction, Proposition~\ref{prop:cl23} shows that the restriction map
\(\rho_{\mathfrak{sl}_2}\) is not surjective. It is therefore the first
explicit finite-rank obstruction visible directly in relative Lie algebra
cohomology. The counting results of
Chang--Lin~\cite[\S 4.3]{CL23} further show that the first
\(\mathfrak{sl}_2\) class outside the multi-graviton sector occurs at level
\(L=24\). Equivalently, there is no fortuitous class for \(\mathfrak{sl}_2\)
at levels \(L<24\), and Proposition~\ref{prop:cl23} identifies the first level
where one appears.
\end{remark}

\begin{remark}[Explicit representatives]
Subsequent papers make Proposition~\ref{prop:cl23} explicit at the level of
representatives. To distinguish a concrete cocycle from the generic cochain
notation \(\mathcal O_c\) in~\eqref{eq:q-minus-d-main}, write
\(\Xi_f^{\mathfrak{sl}_2}\) for one such fortuitous representative.
Choi--Kim--Lee--Park~\cite{CKLP24} express this class in component fields. In
the superspace notation used in the present paper,
Chang--Feng--Lin--Tao~\cite[(5.7)]{CFLT23} give the convenient formula
\[
\begin{aligned}
\Xi_f^{\mathfrak{sl}_2}
={}&
\left(
\operatorname{Tr}\!\bigl(
\Psi_{23}\Psi_1
+
\Psi_{13}\Psi_2
\bigr)
\operatorname{Tr}\!\bigl(
\Psi_{12}\Psi_1
\bigr)
\operatorname{Tr}\!\bigl(
\Psi_{13}
\Psi_{23}
\Psi_{23}
\bigr)
+ \text{cyclic}
\right)\biggr|_{z=\theta=0},
\end{aligned}
\]
where \(\Psi_i:=\partial_{\theta_i}\Psi\),
\(\Psi_{ij}:=\partial_{\theta_i}\partial_{\theta_j}\Psi\), and ``cyclic''
means the sum over cyclic permutations of the indices
\(1,2,3\). The displayed word contains seven factors of \(\Psi\), in agreement
with the cohomological degree in Proposition~\ref{prop:cl23}, and its
derivative multidegree is \((0,0,4,4,4)\). The same paper further shows that
the weak-coupling BPS operator can be written as
\(O_{\mathrm{bh}}=\Xi_f^{\mathfrak{sl}_2}+QO'\), where \(O'\) lies in a
\(17\)-dimensional cyclic-invariant sextic subspace; see~\cite[(5.4),(5.6),(5.8)]{CFLT23}.
Thus the first counterexample is not only detected by counting; it can already
be written as an explicit superspace cocycle.
\end{remark}

\section{Langlands duality and quantum deformation}\label{sec:langlands-quantum}

We compare the relative cohomologies~\eqref{eq:relative-main} for
Langlands-dual Lie algebras. These cohomologies need not agree even at the
classical level. We then record the conjectural formal deformation of the
relative Chevalley--Eilenberg differential, induced by the loop-corrected
supercharge, that is expected to restore Langlands duality.

\begin{remark}[Stable and Cartan-visible comparisons]
For the Langlands-dual pair
\((\mathfrak{so}_{2n+1},\mathfrak{sp}_{2n})\), the stable classical comparison
has no analogue of the finite-rank mismatch below. In the stable picture of
Remark~\ref{rem:osp-dihedral-caveat}, both the orthogonal and symplectic
families are governed by the same reduced dihedral cohomology of \(A\), hence by
the same symmetric algebra after passing to stable relative cohomology.
Likewise, the Cartan-side comparison is tautological: the Weyl groups of
\(\mathfrak{so}_{2n+1}\) and \(\mathfrak{sp}_{2n}\) are naturally identified
with the same hyperoctahedral group. After identifying the Cartan subalgebras,
the target \(\OO(\mathfrak t^{3|2})^W\) of the super-commuting restriction map
is therefore common to both sides. Thus any class compared only through its
Cartan restriction is already identified on this common target. The mismatch
recorded below is consequently a genuinely finite-rank phenomenon, and the
non-Cartan representative lies outside the part visible from the Cartan target.
\end{remark}

\subsection{The classical mismatch: two level-\texorpdfstring{$18$}{18} exceptional classes}\label{subsec:classical-mismatch-level18}

The first explicit classical mismatch occurs for the Langlands-dual pair
\((\soseven,\spsix)\).

\begin{theorem}[Chang--Lin~\cite{CL25}]\label{prop:cl25}
For \(A=\CC[z_+,z_-]\otimes \Lambda(\theta_1,\theta_2,\theta_3)\), the
multigraded cohomologies
\[
H^\bullet(\soseven[A],\soseven;\CC)
\qquad\text{and}\qquad
H^\bullet(\spsix[A],\spsix;\CC)
\]
are not isomorphic as graded vector spaces.
\end{theorem}

\begin{sourceproof}
See~\cite[\S 3]{CL25}. More precisely, Table~1 of~\cite{CL25} shows that the
mismatch first occurs at level \(L=18\) in two different multidegrees. In
cohomological degree \(8\) and derivative multidegree \((0,0,3,3,3)\),
\(\soseven\) has one additional fortuitous class relative to \(\spsix\). The
same table records one additional class in cohomological degree \(8\) and
derivative multidegree \((1,1,2,2,2)\); in the terminology of the present
paper, this is the non-Cartan class. Choi--Lee~\cite[\S 3]{CE25} give an
explicit representative and show that the corresponding non-Coulomb combination
vanishes on the Cartan branch. By~\eqref{eq:top-degree-supercommuting}, the
class lies in the corresponding multihomogeneous piece of
\(\OO(\mathcal C_{\mathfrak g}^{3|2})^G\), where it restricts trivially to the
Cartan side. In particular, the cohomology dimensions differ in explicit
multidegrees.
\end{sourceproof}

\begin{remark}
The same level-by-level comparison in~\cite[Table~1]{CL25} shows that for all
levels \(L<18\), the multigraded cohomologies of \(\soseven\) and \(\spsix\)
agree. Thus Theorem~\ref{prop:cl25} identifies \(L=18\) as the first level
where this Langlands-dual mismatch appears in this example.
\end{remark}

Theorem~\ref{prop:cl25} singles out two exceptional level-\(18\) sectors:
the fortuitous sector of derivative multidegree \((0,0,3,3,3)\) and the
non-Cartan sector of derivative multidegree \((1,1,2,2,2)\). The latter is
also the one relevant to the \(3|2\) super-commuting restriction map of
Section~\ref{sec:super-commuting}, because
equation~\eqref{eq:top-degree-supercommuting} identifies this top-degree
sector with the corresponding multihomogeneous piece of
\(\OO(\mathcal C_{\mathfrak g}^{3|2})^G\). We record explicit
representatives for both classes.

For the formulas below, write
\[
\Psi_i:=\bigl.\partial_{\theta_i}\Psi\bigr|_{z=\theta=0},
\qquad
\Psi_{ij}:=\bigl.\partial_{\theta_i}\partial_{\theta_j}\Psi\bigr|_{z=\theta=0},
\]
and set \(z_{\dot 1}:=z_+,\; z_{\dot 2}:=z_-\). For the bosonic
derivatives, write
\[
\Psi_{\dot\alpha}:=\bigl.\partial_{z_{\dot\alpha}}\Psi\bigr|_{z=\theta=0},
\qquad
\Psi^{\dot\alpha}:=\epsilon^{\dot\alpha\dot\beta}\Psi_{\dot\beta}.
\]
In the fortuitous sector, Gadde--Lee--Raj--Tomar~\cite[(4.23)]{GLRT25} give
the \(\mathfrak{so}_7\) class in the form
\begingroup\small
\begin{equation}\label{eq:glrt-fortuitous-draft}
\begin{aligned}
\Xi_f^{\mathfrak{so}_7}
={}&
\operatorname{Tr}\!\bigl[\Psi_2^2\bigr]
\operatorname{Tr}\!\bigl[\Psi_1\Psi_{23}\bigr]
\operatorname{Tr}\!\bigl[\Psi_1\Psi_3\bigr]^2
\\&
-4\operatorname{Tr}\!\bigl[\Psi_2^2\bigr]
\operatorname{Tr}\!\bigl[\Psi_1\Psi_3\bigr]
\operatorname{Tr}\!\bigl[\Psi_3\Psi_1\Psi_3\Psi_{12}\bigr]
-\operatorname{Tr}\!\bigl[\Psi_1\Psi_3\bigr]^2
\operatorname{Tr}\!\bigl[\Psi_3\Psi_2\Psi_{12}\Psi_2\bigr]
\\&
-4\operatorname{Tr}\!\bigl[\Psi_1\Psi_3\bigr]^2
\operatorname{Tr}\!\bigl[\Psi_3\Psi_2^2\Psi_{12}\bigr]
+8\operatorname{Tr}\!\bigl[\Psi_1\Psi_3\bigr]
\operatorname{Tr}\!\bigl[\Psi_3\Psi_1\Psi_2^2\Psi_3\Psi_{12}\bigr]
\\&
+4\operatorname{Tr}\!\bigl[\Psi_1\Psi_3\bigr]
\operatorname{Tr}\!\bigl[\Psi_3\Psi_1\Psi_3\Psi_2\Psi_{12}\Psi_2\bigr]
+16\operatorname{Tr}\!\bigl[\Psi_1\Psi_3\bigr]
\operatorname{Tr}\!\bigl[\Psi_3\Psi_1\Psi_3\Psi_2^2\Psi_{12}\bigr]
\\&
-4\operatorname{Tr}\!\bigl[\Psi_3\Psi_{12}\bigr]
\operatorname{Tr}\!\bigl[\Psi_3\Psi_1\bigr]
\operatorname{Tr}\!\bigl[\Psi_3\Psi_1\Psi_2^2\bigr]
+8\operatorname{Tr}\!\bigl[\Psi_3\Psi_1\Psi_3\Psi_{12}\bigr]
\operatorname{Tr}\!\bigl[\Psi_1\Psi_2\Psi_3\Psi_2\bigr]
\\&
-2\operatorname{Tr}\!\bigl[\Psi_2^2\bigr]
\operatorname{Tr}\!\bigl[\Psi_3\Psi_{12}\bigr]
\operatorname{Tr}\!\bigl[\Psi_3\Psi_1\Psi_3\Psi_1\bigr]
+8\operatorname{Tr}\!\bigl[\Psi_3\Psi_1\Psi_3\Psi_1\bigr]
\operatorname{Tr}\!\bigl[\Psi_2^2\Psi_3\Psi_{12}\bigr]
\\&
+2\operatorname{Tr}\!\bigl[\Psi_3\Psi_1\Psi_3\Psi_1\bigr]
\operatorname{Tr}\!\bigl[\Psi_2\Psi_3\Psi_2\Psi_{12}\bigr]
+16\operatorname{Tr}\!\bigl[\Psi_2\Psi_3\Psi_{12}\bigr]
\operatorname{Tr}\!\bigl[\Psi_3\Psi_1\Psi_3\Psi_1\Psi_2\bigr]
\\&
+8\operatorname{Tr}\!\bigl[\Psi_3\Psi_{12}\bigr]
\operatorname{Tr}\!\bigl[\Psi_3\Psi_1\Psi_3\Psi_1\Psi_2^2\bigr]
+8\operatorname{Tr}\!\bigl[\Psi_2^2\bigr]
\operatorname{Tr}\!\bigl[\Psi_3\Psi_1\Psi_3\Psi_1\Psi_3\Psi_{12}\bigr]
\\&
+16\operatorname{Tr}\!\bigl[\Psi_3\Psi_1\Psi_3\Psi_1\Psi_2\Psi_3\Psi_2\Psi_{12}\bigr]
\\&
-8\operatorname{Tr}\!\bigl[\Psi_3\Psi_1\Psi_3\Psi_1\Psi_3\Psi_2\Psi_{12}\Psi_2\bigr]
-32\operatorname{Tr}\!\bigl[\Psi_3\Psi_1\Psi_3\Psi_1\Psi_3\Psi_2^2\Psi_{12}\bigr]
\\&
-16\operatorname{Tr}\!\bigl[\Psi_3\Psi_1\Psi_3\Psi_2\Psi_1\Psi_2\Psi_3\Psi_{12}\bigr]
-16\operatorname{Tr}\!\bigl[\Psi_3\Psi_1\Psi_3\Psi_2\Psi_1\Psi_3\Psi_2\Psi_{12}\bigr]
-16\operatorname{Tr}\!\bigl[\Psi_3\Psi_1\Psi_3\Psi_2^2\Psi_1\Psi_3\Psi_{12}\bigr].
\end{aligned}
\end{equation}
\endgroup
This class lies in cohomological degree \(8\) and derivative multidegree
\((0,0,3,3,3)\), hence at level \(18\). It is therefore an explicit
representative of the fortuitous class detected in
Theorem~\ref{prop:cl25}.

In the non-Cartan sector, Choi--Lee~\cite{CE25} give the corresponding class
in component form, and the formula below is its translation into the
superfield notation of Section~\ref{sec:bps-superfields}:
\begin{equation}\label{eq:non-cartan-draft}
\begin{aligned}
\Xi_{nc}^{\mathfrak{so}_7}
={}&
\epsilon_{ijk}\epsilon_{lmn}
\Bigl(
\operatorname{Tr}(\Psi_i\Psi_l)\operatorname{Tr}(\Psi_{\dot\alpha}\Psi^{\dot\alpha})
-\operatorname{Tr}(\Psi_i\Psi_{\dot\alpha})\operatorname{Tr}(\Psi_l\Psi^{\dot\alpha})
-4\operatorname{Tr}(\Psi_{\dot\alpha}\Psi^{\dot\alpha}\Psi_{(i}\Psi_{l)})
\Bigr)
\\
&\qquad\qquad
\times
\operatorname{Tr}(\Psi_j\Psi_m)\operatorname{Tr}(\Psi_k\Psi_n).
\end{aligned}
\end{equation}
This class lies in cohomological degree \(8\) and derivative multidegree
\((1,1,2,2,2)\), hence also at level \(18\). Since it lies in the top-degree
piece identified with \(\OO(\mathcal C_{\mathfrak g}^{3|2})^G\)
by~\eqref{eq:top-degree-supercommuting}, \(\Xi_{nc}^{\mathfrak{so}_7}\) is an
explicit representative of the non-Cartan class detected in
Theorem~\ref{prop:super-chevalley}.

\subsection{The quantum deformation and the conjectural repair}\label{subsec:quantum-deformation-repair}

We state the conjectural quantum deformation expected to repair this classical
mismatch. We first record the charge bookkeeping transported from the physical
complex to the relative cochain complex.

Under the Chang--Yin identification~\eqref{eq:phi-relative}, the relative
cochain space carries the physical charges
\[
(J_1,J_2,q_1,q_2,q_3),
\]
which, classically, are related to the derivative counts and the word length
\(p=\#(\Psi)\) by
\[
J_1=n_{z_-}+\frac{n_{\theta_1}+n_{\theta_2}+n_{\theta_3}-p}{2},
\qquad
J_2=n_{z_+}+\frac{n_{\theta_1}+n_{\theta_2}+n_{\theta_3}-p}{2},
\]
\[
q_1=\frac{p+n_{\theta_1}-n_{\theta_2}-n_{\theta_3}}{2},
\qquad
q_2=\frac{p-n_{\theta_1}+n_{\theta_2}-n_{\theta_3}}{2},
\qquad
q_3=\frac{p-n_{\theta_1}-n_{\theta_2}+n_{\theta_3}}{2},
\]
as follows from~\cite[(4.1) and the following paragraph]{CY13}. In these
conventions, the supercharge \(Q\)
has charge
\[
\left(-\frac12,-\frac12,\frac12,\frac12,\frac12\right)
\]
while the formal parameter \(\hbar\) is taken to have charge \(0\).
Consequently, one expects each correction \(d_i\) to shift charge by the same
amount as \(Q\), rather than preserve all five charges separately. In
particular, the classical count \(p=\#(\Psi)\) is not expected to survive as a
grading in the quantum theory, and the five derivative counts alone no longer
determine the relevant charge sector. A convenient replacement for the
cohomological degree is
\begin{equation}\label{eq:quantum-degree}
\deg:=2(J_1+J_2+q_1+q_2+q_3),
\end{equation}
for which \(Q\) has degree \(+1\). The quantum differential is expected to
preserve the four combinations
\begin{equation}\label{eq:quantum-conserved-combinations}
J_1-J_2,\qquad J_1+q_1,\qquad J_1+q_2,\qquad J_1+q_3.
\end{equation}

\begin{conjecture}[Quantum deformation and Langlands duality]\label{conj:quantum-langlands}
For \(A=\CC[z_+,z_-]\otimes \Lambda(\theta_1,\theta_2,\theta_3)\), there exists
a square-zero deformation of the classical relative Chevalley--Eilenberg
differential
\[
d_{\mathrm{quant}}=d+\hbar d_1+\hbar^2d_2+\cdots
\]
on \(C^\bullet(\mathfrak g[A],\mathfrak g;\CC)[[\hbar]]\) that satisfies
\(d_{\mathrm{quant}}^2=0\), is homogeneous of degree \(+1\) with respect to
\eqref{eq:quantum-degree}, and preserves the four combinations in
\eqref{eq:quantum-conserved-combinations}. Define
\[
H_\hbar(\mathfrak g[A],\mathfrak g)
:=
H\!\bigl(C^\bullet(\mathfrak g[A],\mathfrak g;\CC)[[\hbar]],d_{\mathrm{quant}}\bigr).
\]
Then, for every Langlands-dual pair \((\mathfrak g,{}^L\!\mathfrak g)\), these
quantum-deformed relative cohomologies are isomorphic:
\[
H_\hbar(\mathfrak g[A],\mathfrak g)
\cong
H_\hbar({}^L\!\mathfrak g[A],{}^L\!\mathfrak g)
\]
as \(\CC[[\hbar]]\)-modules, compatibly with the
grading~\eqref{eq:quantum-degree} and the four conserved combinations in
\eqref{eq:quantum-conserved-combinations}.
\end{conjecture}

The expected physical origin of \(d_{\mathrm{quant}}\) is the transport of a
loop-corrected supercharge \(Q_{\mathrm{quant}}\) along the Chang--Yin
identification~\eqref{eq:phi-relative}:
\begin{equation}\label{eq:d-quant-transport}
d_{\mathrm{quant}}=-\,\Phi^{-1}Q_{\mathrm{quant}}\Phi.
\end{equation}
If such a supercharge admits an expansion
\[
Q_{\mathrm{quant}}=Q+\hbar Q_1+\hbar^2Q_2+\cdots
\]
and satisfies \(Q_{\mathrm{quant}}^2=0\), then
\[
d_{\mathrm{quant}}=d+\hbar d_1+\hbar^2d_2+\cdots,
\qquad
d_i:=-\,\Phi^{-1}Q_i\Phi,
\]
and \eqref{eq:d-quant-transport} automatically gives
\(d_{\mathrm{quant}}^2=0\).

\subsection{First quantum evidence for the deformation conjecture}\label{subsec:first-quantum-evidence}

A first concrete piece of evidence for
Conjecture~\ref{conj:quantum-langlands} comes from recent work of
Budzik--Gaiotto--Kulp--Williams--Wu--Yu~\cite{BGKWWY24} and
Budzik--Kulp~\cite[\S 3]{BK26} on the loop-corrected supercharge expected to
produce \(d_{\mathrm{quant}}\) via~\eqref{eq:d-quant-transport}:
\[
Q_{\mathrm{quant}}=Q_0+\hbar Q_1+\hbar^2Q_2+\cdots,
\]
where \(Q_0=Q\) is the classical supercharge introduced in
Section~\ref{sec:bps-superfields}. On the explicit \(\mathfrak{so}_7\)
representatives~\eqref{eq:glrt-fortuitous-draft}
and~\eqref{eq:non-cartan-draft}, Choi--Lee~\cite[\S 3]{CE25} compute the
first quantum correction and obtain the first nontrivial relation, for some
cochain \(\Lambda\),
\begin{equation}\label{eq:q1-pairs}
Q_1\Xi_f^{\mathfrak{so}_7}=\Xi_{nc}^{\mathfrak{so}_7}+Q_0\Lambda.
\end{equation}
Equivalently, \(Q_1\) sends the class of \(\Xi_f^{\mathfrak{so}_7}\) to the
class of \(\Xi_{nc}^{\mathfrak{so}_7}\) in \(Q_0\)-cohomology.
In the charge conventions of Table~1 of~\cite{CL25}, these two classes lie in
the sectors
\[
\left(\frac12,\frac12,\frac52,\frac52,\frac52\right)
\qquad\text{and}\qquad
(0,0,3,3,3)
\]
for \((J_1,J_2,q_1,q_2,q_3)\). Their difference is exactly the charge of \(Q\),
\[
\left(-\frac12,-\frac12,\frac12,\frac12,\frac12\right),
\]
so \(Q_1\) has degree \(+1\) with respect to
\eqref{eq:quantum-degree} and preserves the four combinations in
\eqref{eq:quantum-conserved-combinations}.
Thus the first-order deformed differential already pairs the fortuitous
\(\mathfrak{so}_7\) class~\eqref{eq:glrt-fortuitous-draft} with the non-Cartan
class~\eqref{eq:non-cartan-draft}. In particular, the two exceptional
classical classes in Table~1 of~\cite{CL25} are no longer independent after
the first quantum correction is introduced. At present, however, the cited
literature determines the higher corrections \(d_i\) only indirectly through
transport from the quantum supercharges \(Q_i\); an intrinsic Lie-algebraic
construction of them is still unavailable.

\section{Conclusion}\label{sec:conclusion}

This paper isolates several concrete finite-rank phenomena in relative Lie
algebra cohomology already visible for the supercommutative algebra
\[
A=\CC[z_+,z_-]\otimes \Lambda(\theta_1,\theta_2,\theta_3)
\].
For this algebra, three themes emerge.

The first theme is a super-Chevalley obstruction. The ordinary Chevalley
restriction theorem has commuting-scheme analogues for classical groups, but
the naive \(3|2\) super analogue already fails for \(\mathfrak{so}_7\), and
the obstruction is the explicit non-Cartan class.

The second theme is the appearance of explicit fortuitous classes. The same
algebra \(A\) produces such classes first for \(\mathfrak{sl}_2\) and then
for \(\mathfrak{so}_7\), giving concrete finite-rank counterexamples to naive
stable-image expectations suggested by the type-\(A\)
Loday--Quillen--Tsygan theorem and, for the orthogonal and symplectic families,
by the graded dihedral stable picture of Remark~\ref{rem:osp-dihedral-caveat}.

The third theme is a Langlands-duality mismatch together with a conjectural
repair. The classical relative cohomologies for the Langlands-dual pair
\((\mathfrak{so}_7,\mathfrak{sp}_6)\) are not isomorphic, but recent results
suggest that the \(\mathfrak{so}_7\) mismatch may disappear after a quantum
deformation of the differential, as already indicated by the first relation
\(Q_1\Xi_f^{\mathfrak{so}_7}=\Xi_{nc}^{\mathfrak{so}_7}+Q_0\Lambda\).

From a mathematical point of view, natural concrete problems are to understand
the fortuitous classes more conceptually inside the relative Lie algebra
complex,
to give a direct algebro-geometric proof of the established failure of
\(\operatorname{res}_{\mathfrak g,3|2}\) for \(\mathfrak g=\mathfrak{so}_7\),
and to formulate or construct the conjectural deformation
\(d_{\mathrm{quant}}=d+\hbar d_1+\cdots\) without appealing to the physical
derivation. Independently of the motivating physics conjectures, these
cohomological statements are already explicit and nontrivial enough to merit
further study in the language of Lie algebra cohomology, invariant theory, and
commuting-scheme geometry.

\section*{Acknowledgements}
\addcontentsline{toc}{section}{Acknowledgements}
I thank Ying-Hsuan Lin for collaboration on the high-energy physics work that
provided important motivation and context for this paper, and Penghui Li for
helpful discussions. I am supported by NSFC Grant No.~12575075.

\end{document}